\documentclass{amsart}
\usepackage[latin9]{inputenc}
\usepackage{amstext}
\usepackage{amsthm}
\usepackage{amssymb}

\makeatletter
\numberwithin{equation}{section}
\numberwithin{figure}{section}
\theoremstyle{plain}
\newtheorem{thm}{\protect\theoremname}
  \theoremstyle{definition}
  \newtheorem{defn}[thm]{\protect\definitionname}
  \theoremstyle{remark}
  \newtheorem{rem}[thm]{\protect\remarkname}
  \theoremstyle{definition}
  \newtheorem{example}[thm]{\protect\examplename}
  \theoremstyle{plain}
  \newtheorem{lem}[thm]{\protect\lemmaname}
  \theoremstyle{plain}
  \newtheorem{cor}[thm]{\protect\corollaryname}
  \theoremstyle{plain}
  \newtheorem{prop}[thm]{\protect\propositionname}

\@ifundefined{date}{}{\date{}}
\numberwithin{thm}{section}

\makeatother

  \providecommand{\corollaryname}{Corollary}
  \providecommand{\definitionname}{Definition}
  \providecommand{\examplename}{Example}
  \providecommand{\lemmaname}{Lemma}
  \providecommand{\propositionname}{Proposition}
  \providecommand{\remarkname}{Remark}
\providecommand{\theoremname}{Theorem}

\begin{document}

\title[Invariant operators ]{Operators invariant relative to a completely nonunitary contraction}

\author{H. Bercovici and D. Timotin}

\address{Mathematics Department,\\
Indiana University,\\
Bloomington, IN 47405, USA}

\email{bercovic@indiana.edu}

\address{Simion Stoilow Institute of Mathematics\\
Romanian Academy,\\
Calea Grivi\c tei 21,\\
Bucharest, Romania}

\email{dan.timotin@imar.ro}

\thanks{HB was supported in part by a grant of the National Science Foundation.}


\begin{abstract}
Given a contraction $A$ on a Hilbert space $\mathcal{H}$, an operator
$T$ on $\mathcal{H}$ is said to be $A$-invariant if $\langle Tx,x\rangle=\langle TAx,Ax\rangle$
for every $x\in\mathcal{H}$ such that $\|Ax\|=\|x\|$. In the special
case in which both defect indices of $A$ are equal to $1$, we show
that every $A$-invariant operator is the compression to $\mathcal{H}$
of an unbounded linear transformation that commutes with the minimal
unitary dilation of $A$. This result was proved by Sarason under
the additional hypothesis that $A$ is of class $C_{00}$, leading
to an intrinsic characterization of the truncated Toeplitz operators.
We also adapt to our more general context other results about truncated
Toeplitz operators.
\end{abstract}

\maketitle

\section{Introduction}

Suppose that $A$ is a completely nonunitary contraction acting on
a Hilbert space $\mathcal{H}$ and $U$ is the minimal unitary dilation
of $A$ acting on $\mathcal{K}\supset\mathcal{H}$. Thus, $A^{n}=P_{\mathcal{H}}U^{n}|\mathcal{H}$
is the compression of $U^{n}$ to $\mathcal{H}$ for every positive
integer $n$. It is of interest to consider operators
of the form $P_{\mathcal{H}}X|\mathcal{H}$, where $X$ is in the
commutant $\{U\}'$ of $U$. The commutant lifting theorem \cite{sar-H-infty,harmonic}
shows that every element of $\{A\}'$ is of this form. When $A$ is
the unilateral shift on the Hardy space $H^{2}$, the collection $\{P_{\mathcal{H}}X|\mathcal{H}:X\in\{U\}'\}$
consists precisely of the Toeplitz operators on $H^{2}$. When $A$
is an operator of class $C_{00}$ with defect indices equal to $1$,
the collection $\{P_{\mathcal{H}}X|\mathcal{H}:X\in\{U\}'\}$ is hard
to characterize intrinsically. However, a larger collection, obtained
by considering closed unbounded linear transformations $X$ that commute
with $U$, has been identified in \cite{sar-TTO} with the class of
those bounded operators $Y$ on $\mathcal{H}$ that are $A$-invariant
in the sense that they satisfy the identity
\[
\langle YAx,Ax\rangle=\langle Yx,x\rangle
\]
for every vector $x\in\mathcal{H}$ such that $\|Ax\|=\|x\|$. Of
course, operators $A$ of the type just described can be identified
up to unitary equivalence with compressions of the unilateral shift
to co-invariant subspaces, and the class of operators $Y$ described
above is in that case the class of truncated Toeplitz operators \cite{sar-TTO}.

Our purpose in this paper is to consider arbitrary operators $A$
with defect indices equal to $1$ and the class of bounded operators
on $\mathcal{H}$ that can be obtained as compressions of (possibly)
unbounded linear transformations that commute with $U$. We call these
operators \emph{truncated multiplication operators} and we show, in
particular, that operators in this class are characterized by the
fact that they are $A$-invariant. Operators $A$ with defect indices
equal to $1$ are always complex symmetric and, in the $C_{00}$ case,
it is known \cite{sar-TTO} that the corresponding $A$-invariant
operators satisfy the same complex symmetry. This result no longer
persists if $A$ is not of class $C_{00}$. In this case, the complex
symmetric truncated $A$-invariant operators belong, roughly speaking,
to the linear space generated by $\{A\}'$ and $\{A^{*}\}'$. 

The remainder of the paper is organized as follows. Section \ref{sec:Preliminaries}
contains a description of the functional models of contractions with
defect indices equal to one, as well as the definition of truncated
multiplication operators and their symbols in this context. In Section
\ref{sec:TTO=00003Dinv}, we characterize the class of truncated multiplcation
operators by $A$-invariance. The main result of Section \ref{sec:Nonuniqueness-of-the symbol}
establishes the extent to which the symbol of an $A$-invariant operator
is uniquely determined. In Section \ref{sec:An-analog-of Crofoot}
we describe some useful and explicit unitary equivalences between
model spaces. Finally, in Section \ref{sec:Complex-symmetries} we
discuss complex symmetries, in particular the decomposition of $A$-symmetric
operators into complex symmetric and complex skew-symmetric summands.

\section{Preliminaries\label{sec:Preliminaries}}

We denote by $\mathbb{C}$ the complex plane, by $\mathbb{D}=\{\lambda\in\mathbb{C}:|\lambda|<1\}$
the open unit disk, by $\mathbb{T}=\partial\mathbb{D}$ the unit circle,
and by $\chi$ the identity function $\chi(\lambda)=\lambda$. Normalized
arclength defines a Borel probability measure $m$ on $\mathbb{T}$,
$L^{p}$ stands for the corresponding space $L^{p}(\mathbb{T},m)$,
and $H^{p}\subset L^{p}$ is the Hardy space for $p\in[1,+\infty]$.
We recall that an element $h\in H^{p}$ can also be considered to
be an analytic function on $\mathbb{D}$, and the values of $u$ on
$\mathbb{T}$ can be recovered as radial limits almost everywhere
with respect to $m$.

As noted in the introduction, we focus on contractions $A$ acting
on a Hilbert space $\mathcal{H}$ with the property that the operators
$I_{\mathcal{H}}-A^{*}A$ and $I_{\mathcal{H}}-AA^{*}$ have rank
equal to one. In a different terminology, $T$ has \emph{defect indices}
$1$ and $1$, where the defect indices are a measure of how far $A$
and $A^{*}$ are from being isometric. In addition, we impose the
condition that $A$ has no nonzero reducing subspace $\mathcal{K}$
with the property that the restriction $A|\mathcal{K}$ is a unitary
operator. In other words, $A$ is supposed to be \emph{completely
nonunitary}. 

Sz.-Nagy and Foias have developed a functional model for completely
nonunitary contractions, showing for instance that such a contraction
$A$ is uniquely determined, up to unitary equivalence, by a purely
contractive analytic function $\Theta_{A}$ whose values are operators
between two Hilbert spaces with dimensions equal to the defect indices
of $A$. The function $\Theta_{A}$ is called the \emph{characteristic
function }of $A$, and it plays an analogous role to that of the characteristic
matrix of a linear operator on a finite dimensional space. In our
case, the defect indices are both equal to $1$, so the characteristic
function of $A$ can be thought of simply as a function $u\in H^{\infty}$
such that $\|u\|_{\infty}\le1$. Such a function is purely contractive
precisely when $|u(0)|<1$, that is, when $u$ is not identically
equal to a constant of modulus one. Thus, throughout this paper, we
work with a purely contractive function $u\in H^{\infty}$. When the
characteristic function of $A$ is an inner function in $H^{\infty}$,
the minimal unitary dilation of $A$ is a bilateral shift, and this
allows for the construction of a particularly simple functional model
for $A$. In our more general setting, this dilation is a unitary
operator with spectral multiplicity at most $2$.

We now describe the functional model associated to a given purely
contractive function in $H^{\infty}$. Fix $u\in H^{\infty}$ such
that $\|u\|_{\infty}\le1$ and $|u(0)|<1$, and define the function
$\Delta\in L^{\infty}$ by
\[
\Delta(\zeta)=(1-|u(\zeta)|^{2})^{1/2},\quad\zeta\in\mathbb{T}.
\]
Using this function, we construct spaces
\[
\mathbf{K}=L^{2}\oplus(\Delta L^{2})^{-},\quad\mathbf{K}_{+}=H^{2}\oplus(\Delta L^{2})^{-},\quad\mathbf{G}=\{uf\oplus\Delta f:f\in H^{2}\},
\]
and finally,
\[
\mathbf{H}_{u}=\mathbf{K}_{+}\ominus\mathbf{G}.
\]
Note for further use that a function $f\oplus g\in\mathbf{K}_{+}$
belongs to $\mathbf{H}_u$ if and only if 
\[
\overline{u}f+\Delta g\in L^{2}\ominus H^{2}.
\]
 We define now operators $U\in\mathcal{B}(\mathbf{K}),$ $U_{+}\in\mathcal{B}(\mathbf{K}_{+})$,
and $S_{u}\in\mathcal{B}(\mathbf{H}_{u})$ by
\[
U(f\oplus g)=\chi f\oplus\chi g,\quad f\oplus g\in\mathbf{K},\quad U_{+}=U|\mathbf{K}_{+},
\]
and
\begin{equation}
S_{u}=P_{\mathbf{H}_{u}}U|\mathbf{H}_{u}=(U_{+}^{*}|\mathbf{H}_{u})^{*}.\label{eq:definition of S_u}
\end{equation}
Then the operator $S_{u}$ is completely nonunitary, it has defect
indices equal to $1$, and its characteristic function coincides with
$u$. Moreover $U_{+}$ is the minimal isometric dilation of $S_{u}$,
and $U$ is the minimal unitary dilation of $S_{u}$. We refer to
\cite{harmonic} or \cite{nik} for an exposition of these facts. 

Observe that the operator $S_{u}$ is of class $C_{00}$, that is,
\[
\lim_{n\to\infty}\|S_{u}^{n}h\|=\lim_{n\to\infty}\|S_{u}^{*n}h\|=0,\quad h\in\mathcal{H},
\]
 if and only if $u$ is an inner function, that is, $\Delta=0$. In
this case $\mathbf{H}_{u}=H^{2}\ominus uH^{2}$. In this paper we
concern ourselves primarily with the case in which $u$ is not inner.
All of the arguments in the paper, with the exception of the proof
of Proposition \ref{prop:symbols for the zero operator}, work equally
well if $u$ is an inner function. However, these results were already
known in the inner case. We refer to \cite{sar-TTO} for a detailed
discussion.

We record for further use the formula
\[
U_{+}^{*}(f\oplus g)=\overline{\chi}(f-f(0))\oplus\overline{\chi}g,\quad f\oplus g\in\mathbf{K}_{+}.
\]

We use the linear manifolds
\[
\mathbf{K}^{\infty}=\{f\oplus g:f\in L^{\infty},g\in L^{\infty}\cap(\Delta L^{2})^{-}\},\quad\mathbf{K}_{+}^{\infty}=\mathbf{K_{+}\cap\mathbf{K}}^{\infty},
\]
and
\[
\mathbf{H}_{u}^{\infty}=\mathbf{H}_{u}\cap\mathbf{K}^{\infty}.
\]
It is clear that $\mathbf{K}^{\infty}$ is dense in $\mathbf{K}$
and $\mathbf{K}_{+}^{\infty}$ is dense in $\mathbf{K}_{+}$. To show
that $\mathbf{H}_{u}^{\infty}$ is also dense in $\mathbf{H}_{u}$, we
consider the vectors $\chi^{n}\oplus0$ and $\chi^{-n}u\oplus\chi^{-n}\Delta$,
$n\in\mathbb{Z}.$ These elements of $\mathbf{K}^{\infty}$ span a
dense linear manifold in $\mathbf{K},$ and therefore their orthogonal
projections onto $\mathbf{H}_{u}$ span a dense linear manifold in
$\mathbf{H}_{u}$. These orthogonal projections are again bounded
functions. In fact, $P_{\mathbf{H}_{u}}(\chi^{n}\oplus0)=0$ for $n<0$,
and 
\[
P_{\mathbf{H}_{u}}(\chi^{n}\oplus0)=\chi^{n}\oplus0-P_{\mathbf{G}}(\chi^{n}\oplus0),\quad n\ge0.
\]
The second projection is easily calculated as
\[
P_{\mathbf{G}}(\chi^{n}\oplus0)=\sum_{j=0}^{n}\overline{\alpha_{n-j}}(\chi^{j}u\oplus\chi^{j}\Delta),\quad\alpha_{n-j}=\langle u,\chi^{n-j}\rangle,\quad j=0,\dots,n.
\]
Similarly, $P_{\mathbf{H}_{u}}(\chi^{-n}u\oplus\chi^{-n}\Delta)=0$
for $n\le0$, and 
\[
P_{\mathbf{H}_{u}}(\chi^{-n}u\oplus\chi^{-n}\Delta)=P_{H^{2}}(\chi^{-n}u)\oplus\chi^{-n}\Delta,\quad n>0,
\]
where $P_{H^{2}}:L^{2}\to H^{2}$ denotes the orthogonal projection,
so
\[
P_{H^{2}}(\chi^{-n}u)=\chi^{-n}u-\sum_{j=0}^{n-1}\alpha_{j}\chi^{j-n},\quad\alpha_{j}=\langle u,\chi^{j}\rangle,\quad j=0,\dots,n-1.
\]

Two particularly important vectors in $\mathbf{H}_{u}^{\infty}$ are
defined by
\[
k_{0}=k_{0}^{u}=P_{\mathbf{H}_{u}}(1\oplus0)=(1-\overline{u(0)}u)\oplus(-\overline{u(0)}\Delta)
\]
and
\[
\widetilde{k}_{0}=\widetilde{k}_{0}^{u}=P_{\mathbf{H}_{u}}(\overline{\chi}u\oplus\overline{\chi}\Delta)=(\overline{\chi}(u-u(0))\oplus(\overline{\chi}\Delta).
\]
The operator $S_{u}$ maps $\mathbf{H}_{u}\ominus\mathbb{C}\widetilde{k}_{0}$
isometrically onto $\mathbf{H}_{u}\ominus\mathbb{C}k_{0}$, $S_{u}h=Uh$
for $h\in\mathbf{H}_{u}\ominus\mathbb{C}\widetilde{k}_{0}$, and
\[
S_{u}\widetilde{k}_{0}=-u(0)k_{0},\quad S_{u}^{*}k_{0}=-\overline{u(0)}k_{0}.
\]
 Using these facts and the equalities $\|k_{0}\|^{2}=\|\widetilde{k}_{0}\|^{2}=1-|u(0)|^{2}$,
it is easy to verify the identities
\begin{equation}
I_{\mathbf{H}_{u}}-S_{u}S_{u}^{*}=k_{0}\otimes k_{0},\quad I_{\mathbf{H}_{u}}-S_{u}^{*}S_{u}=\widetilde{k}_{0}\otimes\widetilde{k}_{0},\label{eq:defect of S_u}
\end{equation}
where we use the notation $v\otimes w$ for the rank one operator
$h\mapsto\langle h,w\rangle v$. 

The linear manifolds $\mathbf{K}^{\infty}$ and $\mathbf{K}_{+}^{\infty}$
are clearly invariant under $U$. The linear manifold $\mathbf{H}_{+}^{\infty}$
is also invariant under $S_{u}$, as seen from the formula
\[
S_{u}(f\oplus g)=\chi f\oplus\chi g-\langle f\oplus g,\widetilde{k}_{0}\rangle(u\oplus\Delta),\quad f\oplus g\in\mathbf{H}_{u}.
\]
 Similarly, $\mathbf{H}_{u}^{\infty}$ is invariant under $S_{u}^{*}$
because
\[
S_{u}^{*}(f\oplus g)=\overline{\chi}f\oplus\overline{\chi}g-\langle f\oplus g,k_{0}\rangle(1\oplus0),\quad f\oplus g\in\mathbf{H}_{u}.
\]

It is well known that the commutant $\{U\}'$ consists of multiplication
operators by matrix functions
\[
\left[\begin{array}{cc}
a & b\\
c & d
\end{array}\right],
\]
where $a,b,c,d\in L^{\infty}.$ We require a larger class of unbounded linear transformations that commute with $U$. Suppose
that we are given functions $a,b,c,d\in L^{2}$. We consider the matricial
function
\[
F=\left[\begin{array}{cc}
a & b\\
c & d
\end{array}\right]
\]
and the linear transformation $M_{F}:\mathbf{K}^{\infty}\to\mathbf{K}$
given by
\[
M_{F}(f\oplus g)=(af+bg)\oplus(cf+dg),\quad f\oplus g\in\mathbf{K}^{\infty}.
\]
In other words, $M_{F}$ is the operator of multiplication by $F$.
(Observe that modifying the values of $b,c,$ and $d$ on $\{\zeta\in\mathbb{T}:\Delta(\zeta)=0\}$
does not alter the operator $M_{F}$.  It is useful however to allow
for arbitrary $b,c,d\in L^{2}$.) Generally, $M_{F}$ is not continuous
but it is closable, as can be seen from the inclusion $M_{F^{*}}\subset(M_{F})^{*}$,
where 
\[
F^{*}=\left[\begin{array}{cc}
\overline{a} & \overline{c}\\
\overline{b} & \overline{d}
\end{array}\right].
\]
The equality $M_{F}Uv=UM_{F}v$ holds for every $v\in\mathbf{K}_{u}^{\infty}$.
The operator $M_{F}$ is bounded if and only if $a\in L^{\infty}$
and the functions $b,c,d$ are essentially bounded on $\{\zeta\in\mathbb{T}:\Delta(\zeta)\ne0\}$.
We define a linear transformation $A_{F}:\mathbf{H}_{u}^{\infty}\to\mathbf{H}_{u}$
by
\[
A_{F}v=P_{\mathbf{H}_{u}}M_{F}v,\quad v\in\mathbf{H}_{u}^{\infty}.
\]
 We also have $A_{F^{*}}\subset(A_{F})^{*},$ so $A_{F}$ is always
closable. In particular, $A_{F}$ is bounded if and only if $A_{F}^{*}$
is bounded. If $M_{F}$ is bounded then, of course, $A_{F}$ is bounded
as well, but not conversely \cite{Baranov-c-f-m-t,baranov-b-k}.
\begin{defn}
\label{def:TTO}A bounded linear operator $T\in\mathcal{B}(\mathbf{H}_{u})$
is called a \emph{truncated multiplication operator} if there exists
a function $F$ as above such that $Tv=A_{F}v$ for every $v\in\mathbf{H}_{u}^{\infty}$.
The collection of all truncated multiplication operators is denoted
$\mathcal{T}_{u}$.
\end{defn}
The collection $\mathcal{T}_{u}$ is a linear space, closed under
taking adjoints. In other words, $\mathcal{T}_{u}$ is an \emph{operator
system}.
\begin{rem}
\label{rem:symbol not unique}In the preceding definition, it seems
natural to view $F$ as the \emph{symbol} of the truncated Toeplitz
operator $T$. Note however that there are nonzero functions $F$
such that $A_{F}=0$, and thus a given operator in $\mathcal{T}_{u}$
may have more than one symbol. The symbols $F$ with the property
that $A_{F}=0$ are described in Proposition \ref{prop:symbols for the zero operator}.\end{rem}
\begin{example}
\label{exa:S_u is TMO}The operator $S_{u}$ itself belongs to $\mathcal{T}_{u}$.
One symbol for $S_{u}$ is the function
\[
\left[\begin{array}{cc}
\chi & 0\\
0 & \chi
\end{array}\right],
\]
as can be seen directly from (\ref{eq:definition of S_u}).
\begin{example}
\label{exa:a TTO of rank one}The function
\[
F=\left[\begin{array}{cc}
\chi\overline{u} & \chi\Delta\\
0 & 0
\end{array}\right]
\]
is the symbol of the rank one operator $T=(1-|u(0)|^2)k_0\otimes
\widetilde{k}_0
\in\mathcal{T}_{u}$. To see this, consider
an arbitrary vector $x=h\oplus g\in\mathbf{H}_{u}$, so 
\[
Tx=P_{\mathbf{H}_{u}}(\chi(\overline{u}h+\Delta g)\oplus0)=P_{\mathbf{H}_{u}}(P_{H^{2}}(\chi(\overline{u}h+\Delta g))\oplus0).
\]
It was noted earlier that $\overline{u}h+\Delta g\in L^{2}\ominus H^{2}$,
and therefore $P_{H^{2}}(\chi(\overline{u}h+\Delta g))=P_{H^{2}}(\rho\oplus0)$
for some $\rho\in\mathbb{C}$. The constant $\rho$ is equal to zero
if $x\perp\mathbb{C}\widetilde{k}_{0}$. For $x=k_{0}$, we have
\[
\chi(\overline{u}h+\Delta g)=\overline{u}(u-u(0))+\Delta^{2}=1-u(0)\overline{u},
\]
and thus $P_{H^{2}}(\chi(\overline{u}h+\Delta g))=1-|u(0)|^{2}$.
\end{example}
\end{example}
There is a special class of matrix functions $F$ with the property
that $A_{F}$ commutes with $S_{u}$ on the space $\mathbf{H}_{u}^{\infty}$.
These functions are of the form
\begin{equation}
F=\left[\begin{array}{cc}
a & 0\\
\Delta c & a-uc
\end{array}\right],\label{eq:commutant F}
\end{equation}
where $a\in H^{2}$ and $c\in L^{2}.$ It is easily seen that functions
of this form satisfy $M_{F}(\mathbf{K}_{+}^{\infty})\subset\mathbf{K}_{+}^{\infty}$
and $M_{F}(\mathbf{K}_{+}^{\infty}\cap\mathbf{G})\subset\mathbf{G}$.
Thus, if $x\in\mathbf{H}_{u}^{\infty}$, we have $P_{\mathbf{H}_{u}}(M_{F}P_{\mathbf{G}}Ux)=0$
and $P_{\mathbf{H}_{u}}(UP_{\mathbf{G}}M_{F}x)=0$ and therefore 
\begin{eqnarray*}
A_{F}S_{u}x & = & P_{\mathbf{H}_{u}}M_{F}S_{u}x=P_{\mathbf{H}_{u}}(M_{F}Ux-M_{F}P_{\mathbf{G}}Ux)=P_{\mathbf{H}_{u}}M_{F}Ux\\
 & = & P_{\mathbf{H}_{u}}UM_{F}x=P_{\mathbf{H}_{u}}(UA_{F}x)+P_{\mathbf{H}_{u}}(UP_{\mathbf{G}}M_{F}x)=S_{u}A_{F}x.
\end{eqnarray*}
In the case in which $u\not\equiv0$, the commutant lifting theorem
implies that every bounded operator $T\in\{S_{u}\}'$ is of the form
$A_{F}$, where $F$ is a function of the form \ref{eq:commutant F}
with $a\in H^{\infty}$ and $c\in L^{\infty}$  \cite[Lemma 2.1]{intertwining}.
\begin{lem}
\label{lem:commutant commutative} Suppose that the function $u$
is not identically zero. Then the commutant $\{S_{u}\}'$ is commutative.\end{lem}
\begin{proof}
Suppose that the operators $T,T'\in\{S_{u}\}'$ are determined by
the functions
\[
F=\left[\begin{array}{cc}
a & 0\\
\Delta c & a-uc
\end{array}\right],\quad F'=\left[\begin{array}{cc}
a' & 0\\
\Delta c' & a'-uc'
\end{array}\right],
\]
respectively, for some $a,a'\in H^{\infty}$ and $c\in L^{\infty}$.
A calculation shows that
\[
FF'=F'F=\left[\begin{array}{cc}
a'' & 0\\
\Delta c'' & a''-uc''
\end{array}\right],
\]
where $a''=aa'$ and $c''=ac'+a'c-ucc'$. Suppose that $x$ is an
arbitrary vector in $\mathbf{H}_{u}$. Then
\[
TT'x=P_{\mathbf{H}_{u}}(FP_{\mathbf{H}_{u}}(F'x))=P_{\mathbf{H}_{u}}(FF'x)-P_{\mathbf{H}_{u}}(FP_{\mathbf{G}}(F'x))=P_{\mathbf{H}_{u}}(FF'x),
\]
since $F\mathbf{G}\subset\mathbf{G}\subset\mathbf{H}_{u}^{\perp}$.
It follows that $A_{F''}$ is a symbol for  $TT'$. Similarly, $A_{F''}$
is a symbol for $T'T$, and thus $TT'=T'T$.
\end{proof}
The commutant $\{S_{u}\}'$ is not commutative if $u\equiv0$; see
Example \ref{exa:u=00003D0}.

\section{\label{sec:TTO=00003Dinv}Characterization of truncated multiplication
operators by invariance}

In this section, we show that truncated multiplication operators are
characterized intrinsically by their properties as operators, without
reference to a symbol.
\begin{defn}
\label{def:A-invariant} Suppose that $A$ is a contraction on a Hilbert
space $\mathcal{H}.$ A bounded linear operator $T\in\mathcal{B}(\mathcal{H})$
is said to be \emph{$A$-invariant} if the equality
\[
\langle Tx,y\rangle=\langle TAx,Ay\rangle
\]
holds for every pair of vectors $x,y\in\ker(I_{\mathcal{H}}-A^{*}A)$.\end{defn}
\begin{lem}
\label{lem:about invariance}Suppose that $A\in\mathcal{B}(\mathcal{H})$
is a contraction and $T\in\mathcal{B}(\mathcal{H})$ is an arbitrary
operator. Denote by $\mathcal{D}_{A}=[(I_{\mathcal H}-A^{*}A)\mathcal{H}]^{-}$
and $\mathcal{D}_{A^{*}}=[(I_{\mathcal H}-AA^{*})\mathcal{H}]^{-}$ the defect
spaces of $A$, and by $P_{\mathcal{D}_{A}}$ and $P_{\mathcal{D}_{A^{*}}}$
the corresponding orthogonal projections. Then the following conditions
are equivalent\emph{:}
\begin{enumerate}
\item [\rm(1)]$T$ is $A$-invariant.
\item [\rm(2)]$T$ is $A^{*}$-invariant.
\item [\rm(3)]There exist  $X,Y\in\mathcal{B}(\mathcal{H})$
such that $T-ATA^{*}=XP_{\mathcal{D}_{A^{*}}}+P_{\mathcal{D}_{A^{*}}}Y$.
\item [\rm(4)]There  exist $X,Y\in\mathcal{B}(\mathcal{H})$
such that $T-A^{*}TA=XP_{\mathcal{D}_{A}}+P_{\mathcal{D}_{A}}Y$.
\end{enumerate}
\end{lem}
\begin{proof}
The operator $A$ maps the space $\ker(I-A^{*}A)=\mathcal{D}_{A}^{\perp}$
isometrically onto $\ker(I-AA^{*})=\mathcal{D}_{A^{*}}^{\perp}$.
Thus, given arbitrary vectors $u,v\in\ker(I-AA^{*})$, there exist
unique $x,y\in\ker(I-A^{*}A)$ such that $Ax=u$, $Ay=v$, $A^{*}u=x$,
and $A^{*}v=y$. If $T$ is $A$-invariant, we see that
\[
\langle Tu,v\rangle=\langle TAx,Ay\rangle=\langle Tx,y\rangle=\langle TA^{*}u,A^{*}v\rangle,
\]
and this shows that $T$ is $A^{*}$-invariant. This establishes that
(1) implies (2) and the equivalence of (1) and (2) follows by symmetry.

Suppose now that $T$ is $A$-invariant and observe that
\[
\langle(T-A^{*}TA)x,y\rangle=\langle Tx,y\rangle-\langle TAx,Ay\rangle=0
\]
for every $ x,y\in\ker(I_{\mathcal H }-A^{*}A)=\mathcal{D}_{A}^{\perp}$. It follows that $(I_{\mathcal H }-P_{\mathcal{D}_{A}}) (T-A^*TA)(I_{\mathcal H }-P_{\mathcal{D}_{A}})=0$,
and thus (4) is satisfied with 
\begin{eqnarray*}
X & = & T-A^*TA,\\
Y & = & (T-A^*TA)(I_{\mathcal H }-P_{\mathcal{D}_{A}}).
\end{eqnarray*}
Conversely, if (4) is satisfied, the identity $(I_{\mathcal H }-P_{\mathcal{D}_{A}}) (T-A^*TA)(I_{\mathcal H }-P_{\mathcal{D}_{A}})=0$
follows immediately, thus showing that $T$ is $A$-invariant. We
conclude that (1) is equivalent to (4). The equivalence of (2) and
(3) is proved the same way, replacing $A$ by $A^{*}$.\end{proof}

Suppose now that $u\in H^\infty$ is purely contractive. In the special case of the operator
$A=S_{u}$, (\ref{eq:defect of S_u}) shows that $\ker(I_{\mathbf{H}_{u}}-S_{u}^{*}S_{u})=\mathbf{H}_{u}\ominus\mathbb{C}\widetilde{k}_{0}$.
Moreover, given $x\in\mathbf{H}_{u}$, we have $Ux=S_{u}x\in\mathbf{H}_{u}$
precisely when $x\in\mathbf{H}_{u}\ominus\mathbb{C}\widetilde{k}_{0}$.
Thus an operator $T\in\mathcal{B}(\mathbf{H}_{u})$ is $S_{u}$-invariant
if and only if
\[
\langle Tx,y\rangle=\langle TUx,Uy\rangle,\quad x,y\in\mathbf{H}_{u}\ominus\mathbb{C}\widetilde{k}_{0}.
\]
The polarization identity shows that an operator $T\in\mathcal{B}(\mathbf{H}_{u})$
is $S_{u}$-invariant if and only
\begin{equation}
\langle Tx,x\rangle=\langle TUx,Ux\rangle,\quad x\in\mathbf{H}_{u}\ominus\mathbb{C}\widetilde{k}_{0}.\label{eq:invariance-one}
\end{equation}
The invariance condition can be written equivalently as
\begin{equation}
\langle Tx,x\rangle=\langle TU^{*}x,U^{*}x\rangle,\quad x\in\mathbf{H}_{u}\ominus\mathbb{C}k_{0}.\label{eq:invariance-2}
\end{equation}

We now state the main result in this section.
\begin{thm}
\label{teorema principala}The following four conditions on an operator
$T\in\mathcal{B}(\mathbf{H}_{u})$ are equivalent\emph{:}
\begin{enumerate}
\item [\rm(1)]$T\in\mathcal{T}_{u}$.
\item [\rm(2)]$T$ is $S_{u}$-invariant.
\item [\rm(3)]There exist vectors $v,w\in\mathbf{H}_{u}$ such that $T-S_{u}TS_{u}^{*}=v\otimes k_{0}+k_{0}\otimes w$.
\item [\rm(4)]There exist vectors $\widetilde{v,}\widetilde{w}\in\mathbf{H}_{u}$
such that $T-S_{u}^{*}TS_{u}=\widetilde{v}\otimes\widetilde{k}_{0}+\widetilde{k}_{0}\otimes\widetilde{w}$.
\end{enumerate}
\end{thm}
\begin{proof}
The equations (\ref{eq:defect of S_u}) show that $P_{\mathcal{D}_{S_{u}^{*}}}$
and $P_{\mathcal{D}_{S_{u}}}$ are constant multiples of $k_{0}\otimes k_{0}$
and $\widetilde{k}_{0}\otimes\widetilde{k}_{0}$, respectively. Since
$X(k_{0}\otimes k_{0})=(Xk_{0})\otimes k_{0}$ and $(k_{0}\otimes k_{0})Y=k_{0}\otimes(Y^{*}k_{0})$
for every $X,Y\in\mathcal{B}(\mathbf{H}_{u})$, the equivalence of
(2), (3), and (4) follows immediately from Lemma \ref{lem:about invariance}. 

Suppose now that (1) holds, and thus $Tv=A_{F}v,$ $v\in\mathbf{H}_{u}^{\infty}$,
for some matrix $F$. Since $\widetilde{k}_{0}\in\mathbf{H}_{u}^{\infty}$,
it follows that $\mathbf{H}_{u}^{\infty}\cap$($\mathbf{H}_{u}\ominus\mathbb{C}\widetilde{k}_{0}$)
is dense in $\mathbf{H}_{u}\ominus\mathbb{C}\widetilde{k}_{0}$. Thus,
it suffices to verify (\ref{eq:invariance-one}) for $v\in\mathbf{H}_{u}^{\infty}\cap(\mathbf{H}_{u}\ominus\mathbb{C}\widetilde{k}_{0})$.
For such a vector $v$ we have
\begin{eqnarray*}
\langle TUv,Uv\rangle & = & \langle P_{\mathbf{H}_{u}}M_{F}Uv,Uv\rangle=\langle M_{F}Uv,Uv\rangle=\langle UM_{F}v,Uv\rangle\\
 & = & \langle M_{F}v,v\rangle=\langle P_{\mathbf{H}_{u}}M_{F}v,v\rangle=\langle Tv,v\rangle,
\end{eqnarray*}
where we used the facts that $U$ is unitary and $M_{F}$ commutes
with $U$. We conclude that (1) implies (2).

We come now to the heart of the proof by showing that (3) implies
(1). Suppose that (3) holds for some vectors $v=a_{1}\oplus c$ and
$w=a_{2}\oplus b$ in $\mathbf{H}_{u}$. We define a matrix function
$F$ by
\[
F=\left[\begin{array}{cc}
a_{1}+\overline{a_{2}} & \overline{b}\\
c & 0
\end{array}\right].
\]
We show first that:
\begin{enumerate}
\item [\rm(i)]the operator $A_{F}$ is bounded,
\item [\rm(ii)]the sequence $\{S_{u}^{n}TS_{u}^{*n}\}_{n\in\mathbb{N}}$
converges in the weak operator topology to an operator $T'$ such
that $T'=S_{u}T'S_{u}^{*}$, and
\item [\rm(iii)] $T=A_{F}+T'$ on $\mathbf{H}_{u}^{\infty}$.
\end{enumerate}
To do this, fix a vector $x=g\oplus h\in\mathbf{H}_{u}^{\infty}$
and iterate the relation $T-S_{u}TS_{u}^{*}=v\otimes k_{0}+k_{0}\otimes w$
to obtain
\begin{equation}
Tx=S_{u}^{n}TS_{u}^{*}x+\sum_{j=0}^{n-1}[S_{u}^{j}v\otimes S_{u}^{j}k_{0}+S_{u}^{j}k_{0}\otimes S_{u}^{j}w]x,\quad n\in\mathbb{N}.\label{eq:alta ecuatie}
\end{equation}
We show that the sum above converges weakly to $A_{F}x$. We calculate
first
\begin{eqnarray*}
\sum_{j=0}^{n-1}[S_{u}^{j}k_{0}\otimes S_{u}^{j}w]x & = & \sum_{j=0}^{n-1}\langle x,S_{u}^{j}w\rangle S_{u}^{j}k_{0}=P_{\mathbf{H}_{u}}\sum_{j=0}^{n-1}\langle x,\chi^{j}w\rangle(\chi^{j}\oplus0),
\end{eqnarray*}
where
\[
\langle x,\chi^{j}w\rangle=\langle\overline{a_{2}}g+\overline{b}h,\chi^{n}\rangle.
\]
Therefore, the sum $\sum_{j=0}^{n-1}\langle x,\chi^{j}w\rangle\chi^{j}$
converges in $L^{2}$ to $P_{H^{2}}(\overline{a_{2}}g+\overline{b}h)$
and thus $\sum_{j=0}^{n-1}[S_{u}^{j}k_{0}\otimes S_{u}^{j}w]x$ converges
in norm to $$P_{\mathbf{H}_{u}}(P_{H^{2}}(\overline{a_{2}}g+\overline{b}h)\oplus0)=P_{\mathbf{H}_{u}}((\overline{a_{2}}g+\overline{b}h)\oplus0).$$
Similarly,
\begin{eqnarray*}
\sum_{j=0}^{n-1}[S_{u}^{j}v\otimes S_{u}^{j}k_{0}]x & = & \sum_{j=0}^{n-1}\langle x,S_{u}^{j}k_{0}\rangle S_{u}^{j}v=\sum_{j=0}^{n-1}\langle x,\chi^{j}\oplus0\rangle S_{u}^{j}v\\
 & = & P_{\mathbf{H}_{u}}\left[\sum_{j=0}^{n-1}\langle g,\chi^{j}\rangle\chi^{j}\right]v.
\end{eqnarray*}
Moreover, since 
\[
\sum_{j=0}^{n-1}[S_{u}^{j}v\otimes S_{u}^{j}k_{0}]x=Tx-S_{u}^{n}TS_{u}^{*n}x-\sum_{j=0}^{n-1}[S_{u}^{j}k_{0}\otimes S_{u}^{j}w]x,
\]
it follows that the vectors on the left hand side of this equation
are bounded in $\mathbf{H}_{u}$. To show that they have a weak limit
in $\mathbf{H}_{u}$, it suffices to consider their scalar product
with another element $x'=g'\oplus h'\in\mathbf{H}_{u}^{\infty}$.
We have
\[
\left\langle \sum_{j=0}^{n-1}[S_{u}^{j}k_{0}\otimes S_{u}^{j}w]x,x'\right\rangle =\left\langle \left[\sum_{j=0}^{n-1}\langle g,\chi^{j}\rangle\chi^{j}\right]v,x'\right\rangle ,
\]
and the functions $\sum_{j=0}^{n-1}\langle g,\chi^{j}\rangle\chi^{j}$
converge to $g$ in $H^{2}$ as $n\to\infty$. Since $v\in\mathbf{H}_{u}$
and $x'$ is bounded, the scalar products above tend to $\langle gv,x'\rangle=\langle P_{\mathbf{H}_{u}}(a_{1}g\oplus cg),x'\rangle$
as $n\to\infty$. We conclude that the sum
\[
\sum_{j=0}^{n-1}[S_{u}^{j}v\otimes S_{u}^{j}k_{0}+S_{u}^{j}k_{0}\otimes S_{u}^{j}w]x
\]
converges weakly to $P_{\mathbf{H}_{u}}(((a_{1}+\overline{a_{2}})g+\overline{b}h)\oplus cg)=A_{F}x$.
The identity (\ref{eq:alta ecuatie}) shows that $\|A_{F}x\|\le2\|T\|$,
thus proving (i). Rewriting (\ref{eq:alta ecuatie}) as 
\[
S_{u}^{n}TS_{u}^{*n}=T-\sum_{j=0}^{n-1}[S_{u}^{j}v\otimes S_{u}^{j}k_{0}+S_{u}^{j}k_{0}\otimes S_{u}^{j}w],
\]
we see that $S_{u}^{n}TS_{u}^{*n}x$ converges weakly to $Tx-A_{F}x$
for $x\in\mathbf{H}_{u}^{\infty}$, so the weak convergence of $\{S_{u}^{n}TS_{u}^{*n}\}_{n\in\mathbb{N}}$
follows from the fact that the sequence $\{\|S_{u}^{n}TS_{u}^{*n}\|\}_{n\in\mathbb{N}}$
is bounded. Also, $S_{u}T'S_{u}^{*}$ is the weak limit of the sequence
$\{S_{u}^{n+1}TS_{u}^{*n+1}\}_{n\in\mathbb{N}}$, so it is equal to
$T'$. This proves (ii) and (iii).

To conclude the proof of (1), it suffices to show that $T'\in\mathcal{T}_{u}$.
To do this, we observe that for every $n\in\mathbb{N}$, we have $U_{+}^{*n}\mathbf{H}_{\chi^{n}u}=\mathbf{H}_{u}$
and $U_{+}^{n}\mathbf{H}_{u}\subset\mathbf{H}_{\chi^{n}u}$. We define
an operator $T_{n}\in\mathcal{B}(\mathbf{H}_{\chi^{n}u})$ by
\[
T_{n}x=U_{+}^{n}T'U_{+}^{*n}x,\quad x\in\mathbf{H}_{\chi^{n}u},n\in\mathbb{N}.
\]
Given $n\in\mathbb{N}$ and $x\in\mathbf{H}_{\chi^{n}u}\subset\mathbf{H}_{\chi^{n+1}u}$,
we have
\begin{eqnarray*}
T_{n+1}x & = & U_{+}^{n}U_{+}T'S_{u}^{*}U_{+}^{*n}x\\
 & = & T_{n}x+U_{+}^{n}(U_{+}-S_{u})T'S_{u}^{*}U_{+}^{*n}x.
\end{eqnarray*}
The vector $U_{+}^{n}(U_{+}-S_{u})T'S_{u}^{*}U_{+}^{*n}x$ belongs
to $\mathbf{H}_{\chi^{n+1}u}^{\perp}$, and thus
\[
T_{n}x=P_{\mathbf{H}_{\chi^{n}u}}T_{n+1}x,\quad x\in\mathbf{H}_{\chi^{n}u}.
\]
In particular, $T'=P_{\mathbf{H}_{u}}T_{n}|\mathbf{H}_{u}$ for every
$n\in\mathbb{N}.$ Since $\|T_{n}\|\le\|T\|$, $n\in\mathbb{N}$,
it follows that there exists an operator $X\in\mathcal{B}(\mathbf{K}_{+})$
with the property that $T_{n}=P_{\mathbf{H}_{\chi^{n}u}}X|\mathbf{H}_{\chi^{n}u}$
for every $n\in\mathbb{N}$. In fact, $\bigcup_{m\in\mathbb{N}}\mathbf{H}_{\chi^{m}u}$
is dense in $\mathbf{K}_{+}$, and
\[
Xx=\lim_{n\to\infty}T_{n}x
\]
if $x\in\mathbf{H}_{\chi^{m}u}$ for some $m\in\mathbb{N}$. The operator
$X$ satisfies the identity $X=U_{+}XU_{+}^{*}.$ This implies that
\[
X(g\oplus0)=\lim_{n\to\infty}U_{+}^{n}XU_{+}^{*n}(g\oplus0)=0,\quad g\in H^{2}.
\]
Analogously, the equality $X^{*}=U_{+}X^{*}U_{+}^{*}$ yields $X^{*}(g\oplus0)=0$
for $g\in H^{2}$. Thus, $X$ is of the form $X=0_{H^{2}}\oplus Y$,
where $Y\in\mathcal{B}((\Delta L^{2})^{-})$. Since $X$ commutes
with $U_{+}$, it follows that $Y$ commutes with multiplication by
$\chi$. Thus $Y$ must be the operator of multiplication by some
bounded measurable function $d$, and therefore
\[
X(g\oplus h)=0\oplus dh,\quad g\oplus h\in\mathbf{K}_{+}.
\]
The equality $T'=P_{\mathbf{H}_{u}}X|\mathbf{H}_{u}$ shows that $T'$
is a truncated multiplication operator with symbol
\[
\left[\begin{array}{cc}
0 & 0\\
0 & d
\end{array}\right].
\]
Putting these facts together, we have shown that
\[
T=A_{F}+T'=A_{F'},
\]
where 
\[
F'=\left[\begin{array}{cc}
a_{1}+\overline{a_{2}} & \overline{b}\\
c & 0
\end{array}\right]+\left[\begin{array}{cc}
0 & 0\\
0 & d
\end{array}\right].
\]
We have  established that (3) implies (1), thus concluding
the proof.\end{proof}
\begin{cor}
\label{cor:WOT closed system}For every purely contractive function
$u\in H^{\infty}$, the operator system $\mathcal{T}_{u}$ is closed
in the weak operator topology.\end{cor}
\begin{proof}
By Theorem \ref{teorema principala}, membership of an operator $T$
in $\mathcal{T}_{u}$ is characterised by the system of equations
\[
\langle TS_{u}x,S_{u}x\rangle=\langle Tx,x\rangle,\quad x\in\mathbf{H}_{u}\ominus\mathbb{C}\widetilde{k}_{0}.
\]
Each of these equations is given by a continuous linear functional
in the weak operator topology.\end{proof}
\begin{example}
\label{exa:the rank one perturbations}Given an arbitrary scalar $\mu\in\mathbb{C}$,
we define a bounded linear operator $X_{\mu}\in\mathcal{B}(\mathbf{H}_{u})$
by 
\[
X_{\mu}x=\begin{cases}
Ux=S_{u}x, & x\in\mathbf{H}_{u}\ominus\mathbb{C}\widetilde{k}_{0},\\
\mu k_{0}, & x=\widetilde{k}_{0}.
\end{cases}
\]
It is easily verified using Theorem \ref{teorema principala}(2) that
$X_{\mu}$ is a truncated multiplication operator. We show in Corollary
\ref{cor:commutants a la Sedlock} that the commutant of $X_{\mu}$
consists entirely of truncated multiplication operators.These rank
one perturbations of $S_{u}$ have been considered earlier in \cite{clark}
(when $u$ is inner) and \cite{Ball-Lubin} (see also \cite{livsic,brang}).
The following result follows from \cite{Ball-Lubin}.\end{example}
\begin{prop}
\label{prop:about the various rank one perturbations}Fix $\mu\in\mathbb{C}$,
a purely contractive function $u\in H^{\infty}$, and let $X_{\mu}$
be defined as in Example\emph{ \ref{exa:the rank one perturbations}}.
Then\emph{:}
\begin{enumerate}
\item [\rm(1)]For $|\mu|<1$, the operator $X_{\mu}$ is a completely nounitary
contraction with defect indices equal to $1$.
\item [\rm(2)]For $|\mu|>1$, the operator $X_{\mu}$ is invertible and
$X_{\mu}^{-1}$ is a completely nounitary contraction with defect
indices equal to $1$.
\item [\rm(3)]For $|\mu|=1$, the operator $X_{\mu}$ is unitary with spectral
multiplicity equal to $1$.
\end{enumerate}
\end{prop}
\begin{cor}
\label{cor:commutant of perturbation} With the notation of Proposition
\emph{\ref{prop:about the various rank one perturbations}}, the commutant
of the operator $X_{\mu}$ is commutative for all $\mu\in\mathbb{C}\setminus\{0\}$.
The commutant of $X_{0}$ is also commutative if $u$ is not a constant
function.\end{cor}
\begin{proof}
If $|\mu|\ne1$, the corollary follows from parts (1) and (2) of Proposition
\ref{prop:about the various rank one perturbations} and from Lemma
\ref{lem:commutant commutative}. The case $|\mu|=1$ is a consequence
of the general description of commutants of normal operators. The
case $\mu=0$ follows from the fact that the characteristic function
of $X_{0}$ is zero precisely when $u$ is a constant function.\end{proof}
\begin{cor}
\label{cor:commutants a la Sedlock} Let $\mu\in\mathbb{C}$, and
let $X_{\mu}\in\mathcal{B}(\mathbf{H}_{u})$ be the operator defined
in Example\emph{ \ref{exa:the rank one perturbations}}. Then every
operator $T\in\mathcal{B}(\mathbf{H}_{u})$ that commutes with either
$X_{\mu}$ or with $X_{\mu}^{*}$ is a truncated multiplication operator.\end{cor}
\begin{proof}
We observe first that $\langle X_{\mu}h,Uk\rangle=\langle h,k\rangle$
if $h,k,Uk\in\mathbf{H}_{u}$. This is immediate if $Uh\in\mathbf{H}_{u}$
as well. On the other hand, if $h=\widetilde{k}_{0}$, then $\langle X_{\mu}h,Uk\rangle=\langle h,k\rangle=0$.
Suppose now that $TX_{\mu}=X_{\mu}T$ and $k,Uk\in\mathbf{H}_{u}$.
Then
\[
\langle TUk,Uk\rangle=\langle TX_{\mu}k,Uk\rangle=\langle X_{\mu}Tk,Uk\rangle=\langle Tk,k\rangle,
\]
by the preceding observation applied to $h=Tk$. Thus $T$ is $S_{u}$-invariant
and $T\in\mathcal{T}_{u}$ by Theorem \ref{teorema principala}. If
$TX_{\mu}^{*}=X^{*}T$ then the above argument shows that $T_{\mu}^{*}\in\mathcal{T}_{u}$
and thus $T\in\mathcal{T}_{u}$ because $\mathcal{T}_{u}$ is a selfadjoint
space.\end{proof}
\begin{rem}
\label{rem:Sedlock algebras} In the case in which $u$ is an inner
function, it was shown in \cite{sedlo} that every algebra contained
$\mathcal{T}_{u}$ is commutative and is contained either in $\{X_{\mu}\}'$ or in $\{X_{\mu}^{*}\}'$
for some $\mu\in\mathbb{C}.$ It would be interesting to see whether
this result remains true if $u$ is not inner. Note, incidentally,
that $\mathcal{T}_{u}$ does contain a noncommutative algebra if $u$
is a constant function, namely the commutant of $X_{0}$ (see Example
\ref{exa:u=00003D0}).
\begin{rem}
\label{rem:deB-Rov}In case $u$ is an extreme point of the unit ball
of $H^{\infty}$, it is known (see, for instance, \cite[Chapter IV]{sub-hardy})
that the projection onto the first component yields a unitary operator
$J:\mathbf{H}_{u}\to\mathcal{H}(u)$, where $\mathcal{H}(u)$ is the
de Branges-Rovnyak space associated to $u$. The operator $X=JS_{u}^{*}J^{*}$
is precisely the restriction to $\mathcal{H}(u)$ of the backward
shift $f\mapsto\overline{\chi}(f-f(0))$. Therefore, Theorem \ref{teorema principala}
yields a characterization of those operators in $\mathcal{B}(\mathcal{H}(u))$
that are $X$-invariant.
\end{rem}
\end{rem}

\section{\label{sec:Nonuniqueness-of-the symbol}Nonuniqueness of the symbol
of a truncated multiplication operator}

As noted earlier, the symbol of an operator in $\mathcal{T}_{u}$
is not unique. The proof of Theorem \ref{teorema principala} shows
that a certain sequence related with an operator $T\in\mathcal{T}_{u}$
converges in the weak operator topology. The following result identifies
that limit in terms of an arbitrary symbol for $T$.
\begin{prop}
\label{prop:limit for uniqueness}Suppose that $T\in\mathcal{B}(\mathbf{H}_{u})$
is a truncated multiplication operator with symbol
\[
\left[\begin{array}{cc}
a & b\\
c & d
\end{array}\right].
\]
Then $d$ is essentially bounded on $\{\zeta\in\mathbb{T}:\Delta(\zeta)\ne0\}$,
and the sequence $\{S_{u}^{n}TS_{u}^{*n}\}_{n\in\mathbb{N}}$ converges
in the weak operator topology to the truncated multiplication operator
with symbol
\[
\left[\begin{array}{cc}
0 & 0\\
0 & d
\end{array}\right].
\]
 In particular, the function $d$ is uniquely determined almost everywhere
on $\{\zeta\in\mathbb{T}:\Delta(\zeta)\ne0\}$.\end{prop}
\begin{proof}
Let $x=g\oplus h$ and $x'=g'\oplus h'$ be two vectors in $\mathbf{H}_{u}$.
We have
\[
\langle S_{u}^{n}TS_{u}^{*n}x,x'\rangle=\langle TS_{u}^{*n}x,S_{u}^{*n}x'\rangle,\quad n\in\mathbb{N},
\]
and $S_{u}^{*n}x=P_{+}(\overline{\chi}^{n}g)\oplus\overline{\chi}^{n}h$,
where $P_{+}:L^{2}\to H^{2}$ denotes the orthogonal projection. By
the M. Riesz theorem, $P_{+}$ also defines a bounded operator on $L^{p}$
for $p\in(2,+\infty)$. We have $g\in L^{\infty}\subset L^{6},$ $\lim_{n\to\infty}\|P_{+}(\overline{\chi}^{n}g)\|_{2}=0$,
and an application of the H\"older inequality shows that
\[
\|P_{+}(\overline{\chi}^{n}g)\|_{4}\le\|P_{+}(\overline{\chi}^{n}g)\|_{2}^{1/4}\|P_{+}(\overline{\chi}^{n}g)\|_{6}^{3/4},\quad n\in\mathbb{N}.
\]
We deduce that $\lim_{n\to\infty}\|P_{+}(\overline{\chi}^{n}g)\|_{4}=0$.
Similarly, $\lim_{n\to\infty}\|P_{+}(\overline{\chi}^{n}g')\|_{4}=0$.
Expand now
\begin{eqnarray*}
\langle TS_{u}^{*n}x,S_{u}^{*n}x'\rangle & = & \langle aP_{+}(\overline{\chi}^{n}g),P_{+}(\overline{\chi}^{n}g')\rangle+\langle b\overline{\chi}^{n}h,P_{+}(\overline{\chi}^{n}g')\rangle\\
 &  & +\langle cP_{+}(\overline{\chi}^{n}g),\overline{\chi}^{n}h'\rangle+\langle d\overline{\chi}^{n}h,\overline{\chi}^{n}h'\rangle.
\end{eqnarray*}
The fourth term on the right hand side is equal to $\langle dh,h'\rangle=\langle P_{\mathbf{H}_{u}}(0\oplus dh),x'\rangle$
for every $n\in\mathbb{N}$, and we show that the remaining three
terms converge to zero as $n\to\infty$. The H\"older inequality
yields
\begin{eqnarray*}
|\langle aP_{+}(\overline{\chi}^{n}g),P_{+}(\overline{\chi}^{n}g')\rangle| & \le & \|a\|_{2}\|P_{+}(\overline{\chi}^{n}g)\|_{4}\|P_{+}(\overline{\chi}^{n}g')\|_{4},\\
|\langle b\overline{\chi}^{n}h,P_{+}(\overline{\chi}^{n}g')\rangle| & \le & \|b\|_{2}\|h\|_{4}\|cP_{+}(\overline{\chi}^{n}g')\|_{4},\\
|\langle cP_{+}(\overline{\chi}^{n}g),\overline{\chi}^{n}h'\rangle| & \le & \|c\|_{2}\|cP_{+}(\overline{\chi}^{n}g)\|_{4}\|h'\|_{4},
\end{eqnarray*}
and the sequences in the right hand side tend to zero, as shown above.
We also see that $|\langle dh,h'\rangle|\le\|T\|\|h\|_{2}\|h'\|_{2}$. 

To conclude the proof, we deduce from this inequality that $d$ is
essentially bounded on $\{\zeta\in\mathbb{T}:\Delta(\zeta)\ne0\}$.
We observe that
\[
U_{+}^{*n}(u\oplus\Delta)=P_{+}(\overline{\chi}^{n}u)\oplus\overline{\chi}^{n}\in\mathbf{H}_{u}^{\infty},
\]
and thus $|\langle dh,h'\rangle|\le\|T\|\|h\|_{2}\|h'\|_{2}$ if $h$
and $h'$ are of the form
\[
h=\sum_{j=1}^{n}\alpha_{j}\overline{\chi}^{j},\quad h'=\sum_{j=1}^{n}\alpha'_{j}\overline{\chi}^{j}
\]
 for some $n\in\mathbb{N}$ and $\alpha_{1},\dots,\alpha_{n},\alpha'_{1},\dots,\alpha'_{n}\in\mathbb{C}$.
Moreover, since $\langle dh,h'\rangle=\langle d\chi^{m}h,\chi^{m}h'\rangle$
for every $m\in\mathbb{N},$ we must have $|\langle dh,h'\rangle|\le\|T\|\|h\|_{2}\|h'\|_{2}$
whenever $h=p\Delta$ and $h'=q\Delta$ for some trigonometric polynomials
$p$ and $q$. Since the trigonometric polynomials form a dense linear
manifold in $L^{2}$, we see that 
\begin{equation}
|\langle df\Delta,g\Delta\rangle|\le\|T\|\|f\Delta\|_{2}\|g\Delta\|_{2}\label{eq:o inegalitate}
\end{equation}
 for every pair $f,g$ of functions in $L^{2}$. This finally implies
that $|d|\le\|T\|$ almost everywhere on $\{\zeta\in\mathbb{T}:\Delta(\zeta)\ne0\}$.
Indeed, in the contrary case, there exist $\varepsilon,M>0$ such
that $\|T\|+\varepsilon\le|d|\le M$ on a set $\sigma\subset\{\zeta\in\mathbb{T}:\Delta(\zeta)\ne0\}$
of positive arclength. Then the choice $f=\overline{d}\chi_{\sigma},g=1$
contradicts (\ref{eq:o inegalitate}).
\end{proof}
We can now describe all the symbols associated to the zero operator.
\begin{prop}
\label{prop:symbols for the zero operator}Suppose that $T\in\mathcal{B}(\mathbf{H}_{u})$
is a truncated multiplication operator with symbol
\[
\left[\begin{array}{cc}
a & b\\
c & d
\end{array}\right],
\]
where $a,b,c\in L^{2}$ and $d\in L^{\infty}$. Then $T=0$ if and
only the following two conditions are satisfied\emph{:}
\begin{enumerate}
\item [\rm(1)]$d=0$ almost everywhere on the set $\{\zeta\in\mathbb{T}:\Delta(\zeta)\ne0\}$.
\item [\rm(2)]There exist functions $f_{1},f_{2}\in H^{2}$ such that\emph{:}

\begin{enumerate}
\item [\rm(a)]$a=uf_{1}+\overline{uf_{2}},$
\item [\rm(b)]$c=\Delta f_{1}$ and $b=\Delta\overline{f}_{2}$ almost
everywhere on the set $\{\zeta\in\mathbb{T}:\Delta(\zeta)\ne0\}$.
\end{enumerate}
\end{enumerate}
\end{prop}
\begin{proof}
The case in which $u$ is an inner function is proved in \cite[Theorem 3.1]{sar-TTO}.
Therefore we may, and do, assume that $u$ is not inner, and thus the set $\{\zeta\in\mathbb{T}:\Delta(\zeta)\ne0\}$
has positive arclength.

Suppose first that conditions (1) and (2) are satisfied and set
\begin{equation}
F_{1}=\left[\begin{array}{cc}
uf_{1} & 0\\
\Delta f_{1} & 0
\end{array}\right],\quad F_{2}=\left[\begin{array}{cc}
uf_{2} & 0\\
\Delta f_{2} & 0
\end{array}\right],\label{eq:F_j}
\end{equation}
so $F=F_{1}+F_{2}^{*}$. If $x=h\oplus g$ is an arbitrary element
of $\mathbf{H}_{u}^{\infty},$ we have 
\[
A_{F_{1}}x=uf_{1}h\oplus\Delta f_{1}h\in\mathbf{G},
\]
and thus $A_{F_{1}}=0$. Similarly, $A_{F_{2}}^{*}=0$, and therefore
$T|\mathbf{H}_{u}^{\infty}=A_{F}=0$. 

Conversely, suppose that $T=0$. Condition (1) follows from Proposition
\ref{prop:limit for uniqueness}. In addition, we have $Tk_{0}=T^{*}k_{0}=0$.
Since $k_{0}=(1-\overline{u(0)}u)\oplus(-\overline{u(0)}\Delta)$,
the vectors
\begin{eqnarray*}
Fk_{0} & = & [a(1-\overline{u(0)}u)-b\overline{u(0)}\Delta]\oplus[c(1-\overline{u(0)}u)],\\
F^{*}k_{0} & = & [\overline{a}(1-\overline{u(0)}u)-\overline{cu(0)}\Delta]\oplus[\overline{b}(1-\overline{u(0)}u)],
\end{eqnarray*}
must belong to $\mathbf{H}_{u}^{\perp}=[H^{2\perp}\oplus\{0\}]+\mathbf{G},$
that is,
\begin{eqnarray*}
Fk_{0} & = & (g_{1}+uh_{1})\oplus(\Delta h_{1}),\\
F^{*}k_{0} & = & (g_{2}+uh_{2})\oplus(\Delta h_{2}),
\end{eqnarray*}
for some $g_{1},g_{2}\in H^{2\perp}$ and $h_{1},h_{2}\in H^{2}$.
Equating the second components, we see that
\[
c=\frac{\Delta h_{1}}{1-\overline{u(0)}u}=\Delta f_{1},\quad\overline{b}=\frac{\Delta h_{2}}{1-\overline{u(0)}u}=\Delta f_{2},
\]
where $f_{j}=h_{j}/(1-\overline{u(0)}u)\in H^{2}$ for $j=1,2$. Define
now $F_{1}$ and $F_{2}$ by (\ref{eq:F_j}) and set 
\[
F_{0}=F-F_{1}-F_{2}^{*}=\left[\begin{array}{cc}
a_{0} & 0\\
0 & 0
\end{array}\right],
\]
where $a_{0}=a-uf_{1}-\overline{uf_{2}}$. The hypothesis and the
first part of the proof show that $F_{0}$ is also a symbol of the
zero operator, and thus the vectors
\[
F_{0}k_{0}=a_{0}(1-\overline{u(0)}u)\oplus0,\quad F_{0}^{*}k_{0}=\overline{a_{0}}(1-\overline{u(0)}u),
\]
must belong to $\mathbf{H}_{u}^{\perp}$. Observe that, given $f\in H^{2}$,
the equality $\Delta f=0$ implies that $f$ vanishes almost everywhere
on the set $\{\zeta\in\mathbb{T}:\Delta(\zeta)\ne0\}$, and thus $f=0$
by the F. and M. Riesz theorem. Therefore there exist functions $g_{1},g_{2}\in H^{2\perp}$
such that
\[
a_{0}(1-\overline{u(0)}u)=g_{1},\quad\overline{a_{0}}(1-\overline{u(0)}u)=g_{2}.
\]
We have then
\[
a_{0}=\frac{g_{1}}{1-\overline{u(0)}u}=\frac{\overline{g_{2}}}{1-u(0)\overline{u}},
\]
so
\[
g_{1}(1-u(0)\overline{u})=\overline{g_{2}}(1-\overline{u(0)}u).
\]
The left hand side of this equality has a Fourier series with no analytic
terms, while the right hand side has only analytic terms. We conclude
that $g_{1}=g_{2}=0$, thus establishing that $F_{0}=0$. The proposition
follows.\end{proof}
\begin{example}
\label{exa:u=00003D0}We examine the special case $u=0$. In this
case, $\mathbf{K}_{+}=H^{2}\oplus L^{2}$, $\mathbf{G}=\{0\}\oplus H^{2}$,
 $\mathbf{H}_{u}=H^{2}\oplus(L^{2}\ominus H^{2})$, and thus the
operator $S_{u}$ is of the form $A\oplus B$, where $A$ is the forward
shift on $H^{2}$ and $B$ is the forward (co-isometric) shift on
$L^{2}\ominus H^{2}$. (In other words, $S_{u}$ is unitarily equivalent
to $A\oplus A^{*}$.) A symbol
\[
F=\left[\begin{array}{cc}
a & b\\
c & d
\end{array}\right]
\]
represents the zero operator in $\mathcal{T}_{u}$ precisely when
$a=d=0$ and $\overline{b},c\in H^{2}$. In particular, the (1,1)
and (2,2) entries of the symbol of an operator $T\in\mathcal{T}_{u}$
are uniquely determined by $T$. The operators that commute with $S_{u}$
are described, using the commutant lifting theorem, as the truncated
multiplication operators with a symbol of the form 
\[
F=\left[\begin{array}{cc}
a & 0\\
c & d
\end{array}\right],
\]
where $a,d\in H^{\infty}$ and $b\in L^{\infty}$. The commutant of
$S_{u}$ is not commutative, as illustrated by the operators $T_{1}$
and $T_{2}$ with symbols
\[
F_{1}=\left[\begin{array}{cc}
1 & 0\\
0 & 0
\end{array}\right]\text{ and }F_{2}=\left[\begin{array}{cc}
0 & 0\\
\overline{\chi} & 0
\end{array}\right],
\]
for which $T_{1}T_{2}=0$ and $T_{2}T_{1}=T_{2}\ne0$.
\end{example}

\section{An analog of the Crofoot operator\label{sec:An-analog-of Crofoot}}

Suppose that $u\in H^{\infty}$ satisfies $\|u\|_{\infty}\le1$ and
$|u(0)|<1$. The operator $X_{\mu}$ introduced in Example \ref{exa:the rank one perturbations}
is a completely nonunitary contractions with defect indices equal
to $1$ provided that $|\mu|<1$. The caracteristic function of $X_{\mu}$
is equal to
\[
u_{\alpha}=\frac{u-\alpha}{1-\overline{\alpha}u}\in H^{\infty},
\]
where $\alpha\in\mathbb{D}$ is chosen such that $u_{\alpha}(0)=-\mu$.
Thus, there exists a unitary operator in $\mathcal{B}(\mathbf{H}_{u},\mathbf{H}_{u_{\alpha}})$, uniquely determined up to a constant factor of modulus $1$,
that intertwines $X_{\mu}$ and $S_{u_{\alpha}}$. This unitary
operator was first written explicitly by Crofoot \cite{crof} for
the case in which $u$ is inner, and thus $u_{\alpha}$ is inner as
well. He showed that it is the restriction to $\mathbf{H}_{u}$ of
the multiplication operator by a function in $H^{\infty}$. We prove
an analogous result for arbitrary purely contractive functions $u\in H^{\infty}$.
To begin with, a simple calculation shows that the function $\Delta_{\alpha}=(1-|u_{\alpha}|^{2})^{1/2}$
satisfies
\[
\Delta_{\alpha}=\frac{(1-|\alpha|^{2})^{1/2}}{|1-\overline{\alpha}u|}\Delta,
\]
so $(\Delta_{\alpha}L^{2})^{-}=(\Delta L^{2})^{-}$, and therefore
$\mathbf{H}_{u}$ and $\mathbf{H}_{u_{\alpha}}$ are both subspaces
of $\mathbf{K}$. More precisely,
\[
\mathbf{H}_{u_{\alpha}}=\mathbf{K}_{+}\ominus\mathbf{G}_{\alpha},
\]
where
\[
\mathbf{G}_{\alpha}=\{u_{\alpha}f\oplus\Delta_{\alpha}f:f\in H^{2}\}.
\]
We consider the bounded measurable function $F_{\alpha}$ defined
by
\begin{eqnarray*}
F_{\alpha} & = & \left[\begin{array}{cc}
(1-|\alpha|^{2})^{1/2}(1-\overline{\alpha}u)^{-1} & 0\\
\overline{\alpha}\Delta|1-\overline{\alpha}u|^{-1} & (1-\overline{\alpha}u)|1-\overline{\alpha}u|^{-1}
\end{array}\right]\\
 & = & \left[\begin{array}{cc}
(1-|\alpha|^{2})^{1/2}(1-\overline{\alpha}u)^{-1} & 0\\
\overline{\alpha}(1-|\alpha|^{2})^{-1/2}\Delta_{\alpha} & (1-\overline{\alpha}u)|1-\overline{\alpha}u|^{-1}
\end{array}\right].
\end{eqnarray*}
Since the $(1,1)$ entry of $F_{\alpha}$ belongs to $H^{\infty}$,
it follows that $M_{F_{\alpha}}$ leaves $\mathbf{K}_{+}$ invariant. 
\begin{prop}
\label{prop:crofoot}The operator $M_{F_{\alpha}}$ maps $\mathbf{H}_{u}$
isometrically onto $\mathbf{H}_{u_{\alpha}}$.\end{prop}
\begin{proof}
Suppose that $f\oplus g\in\mathbf{H}_{u}$ and thus $\overline{u}f+\Delta g\in L^{2}\ominus H^{2}$.
As noted above, the vector $f_{\alpha}\oplus g_{\alpha}=M_{F_{\alpha}}(f\oplus g)$
belongs to $\mathbf{K}_{+}$. A direct calculation shows that
\[
\overline{u_{\alpha}}f_{\alpha}+\Delta_{\alpha}g_{\alpha}=(1-|\alpha|^{2})^{1/2}(\overline{u}f+\Delta g)(1-\alpha\overline{u})^{-1},
\]
and this function belongs to $L^{2}\ominus H^{2}$ because
$\overline{u}f+\Delta g\in L^2\ominus H^2$ and $(1-\alpha\overline{u})^{-1}$
is a bounded, conjugate analytic function. We conclude that $f_{\alpha}\oplus g_{\alpha}\in\mathbf{H}_{u_{\alpha}}$.
In order to calculate the norm of $f_{\alpha}\oplus g_{\alpha}$ we
observe that 
\[
w=(\overline{\alpha}(1-|\alpha|^{2})^{-1/2}u_{\alpha}f)\oplus(\overline{\alpha}(1-|\alpha|^{2})^{-1/2}\Delta_{\alpha}f)\in\mathbf{G}_{\alpha}
\]
 and thus 
\[
\|f_{\alpha}\oplus g_{\alpha}\|^{2}=\|(f_{\alpha}\oplus g_{\alpha})-w\|^{2}+\|w\|^{2}=\|(f_{\alpha}\oplus g_{\alpha})-w\|^{2}+|\alpha|^{2}(1-|\alpha|^{2})^{-1}\|f\|^{2}.
\]
 Since
\[
(f_{\alpha}\oplus g_{\alpha})-w=((1-|\alpha|^{2})^{-1/2}f)\oplus((1-\overline{\alpha}u)|1-\overline{\alpha}u|^{-1}g),
\]
it follows that $\|(f_{\alpha}\oplus g_{\alpha})-w\|^{2}=(1-|\alpha|^{2})^{-1}\|f\|^{2}+\|g\|^{2}$
and hence that $\|f_{\alpha}\oplus g_{\alpha}\|^{2}=\|f\oplus g\|^{2}.$
The fact that $M_{F_{\alpha}}$ maps $\mathbf{H}_{u}$ onto $\mathbf{H}_{u_{\alpha}}$
follows from the above considerations applied to the operator $M_{F_{\alpha}}^{-1}$
because
\[
F_{\alpha}^{-1}=\left[\begin{array}{cc}
(1-|\alpha|^{2})^{1/2}(1+\overline{\alpha}u_{\alpha}){}^{-1} & 0\\
\overline{\alpha}\Delta_{\alpha}|1-\overline{\alpha}u_{\alpha}|^{-1} & (1-\overline{\alpha}u_{\alpha})|1-\overline{\alpha}u_{\alpha}|^{-1}
\end{array}\right],
\]
and $u=(u_{\alpha}+\alpha)(1+\overline{\alpha}u_{\alpha})^{-1}$.
\end{proof}
We denote by $V_{\alpha}\in\mathcal{B}(\mathbf{H}_{u},\mathbf{H}_{u_{\alpha}})$
the unitary operator defined by $V_{\alpha}x=M_{F_{\alpha}}x$, $x\in\mathbf{H}_{u}$.
In the case in which $u$ is inner, $V_{\alpha}$ is precisely the
operator constructed in \cite{crof}.
\begin{prop}
\label{prop:TTO preserved by Crofoot}An operator $T\in\mathcal{B}(\mathbf{H}_{u})$
is a truncated multiplication operator if and only if $V_{\alpha}TV_{\alpha}^{*}\in\mathcal{B}(\mathbf{H}_{u_{\alpha}})$
is a truncated multiplication operator. Thus, $\mathcal{T}_{u_{\alpha}}=\{V_{\alpha}TV_{\alpha}^{*}:T\in\mathcal{T}_{u}\}$.\end{prop}
\begin{proof}
Fix $T\in\mathcal{B}(\mathbf{H}_{u})$ and define $T_{\alpha}=V_{\alpha}TV_{\alpha}^{*}$,
so
\[
\langle T_{\alpha}V_{\alpha}x,V_{\alpha}x\rangle=\langle Tx,x\rangle,\quad x\in\mathbf{H}_{u}.
\]
 Since $M_{F_{\alpha}}U=UM_{F_{\alpha}}$, it follows that $Ux\in\mathbf{H}_{u}$
if and only if $UV_{\alpha}x\in\mathbf{H}_{u_{\alpha}}$. We conclude
from the preceding identity that $T$ is $U$-invariant if and only
if $T_{\alpha}$ is $U$-invariant. The proposition follows from Theorem
\ref{teorema principala}.
\end{proof}

\section{\label{sec:Complex-symmetries}Complex symmetries}

Suppose that $\mathcal{H}$ is a (complex) Hilbert space. A map $C:\mathcal{H}\to\mathcal{H}$
is called a \emph{conjugation }if it is conjugate linear, isometric,
and $C^{2}=I_{\mathcal{H}}$. A bounded operator $T\in\mathcal{B}(\mathcal{H})$
is said to be $C$-\emph{symmetric} 
(respectively, $C$-\emph{skew-symmetric})
if $CTC=T^{*}$ (respectively, $CTC=-T^{*}$). The operator $T$ is
said to be \emph{complex symmetric }if it is\emph{ $C$-symmetric
}for some conjugation $C$.\emph{ }An operator $T$ can be complex
symmetric relative to several conjugations. For instance, suppose
that $U\in\mathcal{B}(L^{2})$ is the bilateral shift, that is, $Uf=\chi f$,
$f\in L^{2}$. Given an arbitrary function $v\in L^{\infty}$ such
that $|v|=1$ almost everywhere, the formula
\[
C_{v}f=v\overline{f},\quad f\in L^{2},
\]
defines a conjugation on $L^{2}$ such that $U$ is $C_{v}$-symmetric.
(It easy to see that these are all the conjugations relative to which
$U$ is symmetric.)
\begin{prop}
\label{prop:symmetries, not unique} Suppose that $T\in\mathcal{B}(\mathcal{H})$,
and $C$ and $D$ are two symmetries such that $T$ is both $C$-symmetric
and $D$-symmetric. Then at least one of the following is true\emph{:}
\begin{enumerate}
\item [\rm(1)]There exists a constant $\gamma\in\mathbb{T}$ such that
$D=\gamma C$.
\item [\rm(2)]There exists a proper reducing subspace $\mathcal{K}$ for
$T$ such that both $T|\mathcal{K}$ and $T|\mathcal{K}^{\perp}$
are complex symmetric.
\end{enumerate}
\end{prop}
\begin{proof}
Suppose that (1) is not true, and therefore the operator $V=DC$ is
not a scalar multiple of $I_{\mathcal{H}}$. The operator $V$ is
unitary and
\[
VT=DCT=DT^{*}C=TDC=TV.
\]
Moreover, we have
\[
CVC=CD=(DC)^{-1}=V^{-1}=V^{*},
\]
so $V$ is $C$-symmetric. If $E_{V}$ denotes the spectral measure
of $V$, it follows that $E_{V}(\omega)$ is also $C$-symmetric for
every Borel set $\omega\subset\mathbb{T}$, and therefore $E(\omega)TE(\omega)$
is also $C$-symmetric. To show that (2) is true, simply choose $\omega$
such that $0\ne E(\omega)\ne I_{\mathcal{H}}$ and set $\mathcal{K}=E(\omega)\mathcal{H}$.
Then $T|\mathcal{K}$ is $C|\mathcal{K}$-symmetric and $T|\mathcal{K}^{\perp}$
is $C|\mathcal{K}^{\perp}$-symmetric.
\end{proof}
Given a function $u\in H^{\infty}$ such that $\|u\|_{\infty}\le1$
and $|u(0)|<1$, the operator $S_{u}$ does not have any nontrivial
reducing subspaces unless $u=0$. For $u=0$, $S_{u}$ has exactly
one pair of complementary nontrivial reducing subspaces, and the restrictions
of $S_{u}$ to these spaces are a unilateral shift and the adjoint
of a unilateral shift, neither of which is complex symmetric relative to any conjugation. It follows
that, up to a constant multiple of modulus one, there is at most one
conjugation $C$ such that $S_{u}$ is $C$-symmetric. If $u$ is
inner or, more generally, if $u$ is an extreme point of the unit
ball of $H^{\infty}$, it follows from \cite{lot-sar} that $S_{u}$
is complex symmetric (see also \cite{garcia-p}). More general results
about functional models \cite{che-fri-tim} show that $S_{u}$ is
always complex symmetric. We describe below the essentially unique
conjugation $C_{u}$ such that $S_{u}$ is $C_{u}$-symmetric.

The spaces $\mathbf{K},\mathbf{G}$, and $\mathbf{H}_{u}$ in the
following statement were defined in Section \ref{sec:Preliminaries}.
\begin{prop}
Let $u\in H^{\infty}$ be such that $\|u\|_{\infty}\le1$ and $|u(0)|<1$.
Then the operator $C:\mathbf{K}\to\mathbf{K}$ defined by
\[
C(f\oplus g)=(\overline{\chi}u\overline{f}+\overline{\chi}\Delta\overline{g})\oplus(\overline{\chi}\Delta\overline{f}-\overline{\chi}\overline{u}\overline{g}),\quad f\oplus g\in\mathbf{K},
\]
is a conjugation such that $U$ is $C$-symmetric. Moreover, we have
$C\mathbf{H}_{u}=\mathbf{H}_{u}$ and the operator $C_{u}=C|\mathbf{H}_{u}$
is a conjugation such that $S_{u}$ is $C_{u}$-symmetric.\end{prop}
\begin{proof}
The operator $C$ is simply complex conjugation followed by multiplication
by the matrix function
\[
\overline{\chi}\left[\begin{array}{cc}
u & \Delta\\
\Delta & -\overline{u}
\end{array}\right].
\]
It is easily seen that the matrix
\[
\left[\begin{array}{cc}
u(\zeta) & \Delta(\zeta)\\
\Delta(\zeta) & -\overline{u(\zeta)}
\end{array}\right]
\]
is unitary for $\zeta\in\mathbb{T}$, and thus $C$ is an isometry.
The operator $C^{2}$ is the multiplication operator by the matrix
function
\[
\left[\begin{array}{cc}
\overline{u} & \Delta\\
\Delta & -u
\end{array}\right]\left[\begin{array}{cc}
u & \Delta\\
\Delta & -\overline{u}
\end{array}\right]=\left[\begin{array}{cc}
1 & 0\\
0 & 1
\end{array}\right],
\]
and thus $C^{2}=I_{\mathbf{K}}.$ The identity $U^{*}C=CU$ is also
immediate. Observe next that
\[
C(uf\oplus\Delta f)=\overline{\chi}\overline{f}\oplus0,\quad f\in H^{2},
\]
which shows that $C(\mathbf{G})=H^{2}\oplus\{0\}$, and thus $C(H^{2}\oplus\{0\})=\mathbf{G}$
as well. We conclude that $C(\mathbf{H}_{u}^{\perp})=\mathbf{H}_{u}^{\perp}$,
$C(\mathbf{H}_{u})=\mathbf{H}_{u}$, and $C_{u}$ is indeed a conjugation
on $\mathbf{H}_{u}$. Finally,
\[
S_{u}C_{u}=P_{\mathbf{H}_{u}}UC|\mathbf{H}_{u}=P_{\mathbf{H_{u}}}CU^{*}|\mathbf{H}_{u}=CP_{\mathbf{H}_{u}}U^{*}|\mathbf{H}_{u}=C_{u}S_{u}^{*},
\]
showing that $S_{u}$ is $C_{u}$-symmetric.
\end{proof}
We note for further use the equality
\begin{equation}
C_{u}k_{0}=\widetilde{k}_{0}.\label{eq:C of k_0}
\end{equation}

The linear manifold $\mathbf{K}^{\infty}$ is invariant under the
conjugation $C$. It is not the case that every multiplication operator
$M_{F}$ satisfies the equation $M_{F}v=CM_{F^{*}}Cv$ for every $v\in\mathbf{K}^{\infty}$.
\begin{prop}
\label{prop:MF symmetric}The multiplication operator $M_{F}$ by
the matrix function 
\[
F=\left[\begin{array}{cc}
a & b\\
c & d
\end{array}\right]
\]
satisfies the equation $M_{F}=CM_{F^{*}}C|\mathbf{K}^{\infty}$ if
and only if the equality
\begin{equation}
\Delta(d-a)=-uc-\overline{u}b\label{eq:C-symmetric matrix}
\end{equation}
holds almost everywhere on $\{\zeta\in\mathbb{T}:\Delta(\zeta)\ne0\}$.\end{prop}
\begin{proof}
The operator $CM_{F^{*}}C|\mathbf{K}^{\infty}$ is the operator of
multiplication by the matrix
\[
\left[\begin{array}{cc}
u & \Delta\\
\Delta & -\overline{u}
\end{array}\right]\left[\begin{array}{cc}
a & c\\
b & d
\end{array}\right]\left[\begin{array}{cc}
\overline{u} & \Delta\\
\Delta & -u
\end{array}\right],
\]
and a calculation shows that
\[
\left[\begin{array}{cc}
u & \Delta\\
\Delta & -\overline{u}
\end{array}\right]\left[\begin{array}{cc}
a & c\\
b & d
\end{array}\right]\left[\begin{array}{cc}
\overline{u} & \Delta\\
\Delta & -u
\end{array}\right]-F=\left[\begin{array}{cc}
\Delta h & -uh\\
-\overline{u}h & -\Delta h
\end{array}\right],
\]
where
\[
h=\Delta(d-a)+uc+\overline{u}b.
\]
The desired conclusion follows from Proposition \ref{prop:symbols for the zero operator}(1).\end{proof}
\begin{cor}
\label{cor:AF symmetric}If the matrix 
\[
F=\left[\begin{array}{cc}
a & b\\
c & d
\end{array}\right]
\]
 satisfies the equality $\Delta(a-d)=uc+\overline{u}b$ almost everywhere
on $\{\zeta\in\mathbb{T}:\Delta(\zeta)\ne0\}$, then $A_{F}=C_{u}A_{F^{*}}C_{u}|\mathbf{H}_{u}^{\infty}$.
\end{cor}
In the particular case in which $u$ is an inner function, the function
$\Delta$ is equal to zero almost everywhere. Thus, the preceding
corollary shows that every operator in $\mathcal{T}_{u}$ is $C_{u}$-symmetric.
This result \cite[Section 2.3]{sar-TTO} plays an important role in
the study of truncated Toeplitz operators. If $u$ is not inner, there
are operators in $\mathcal{T}_{u}$ that are not $C_{u}$-symmetric.
For instance, the operator with symbol
\[
\left[\begin{array}{cc}
0 & 0\\
0 & 1
\end{array}\right]
\]
is not $C_{u}$-symmetric.

Suppose that $T\in\mathcal{B}(\mathbf{H}_{u})$ is a truncated multiplication
operator. Then the operator $C_{u}T^{*}C_{u}$ is easily seen to be a truncated
multiplication operator as well. It follows that $T$ can be written in a
unique way as a sum $T=T_{1}+T_{2}$, where $T_{1}=(1/2)(T+C_{u}T^{*}C_{u})$
is a $C_{u}$-symmetric truncated multiplication operator and $T_{2}=(1/2)(T-C_{u}T^{*}C_{u})$
is a $C_{u}$-skew-symmetric operator. The above calculations allow
us to show that the operators $T_{1}$ and $T_{2}$ have symbols of
a special form.
\begin{prop}
\label{prop:symbols of symmetric and skew-symmetric TTO}Suppose that
$T\in\mathcal{T}_{u}$. Then\emph{:}
\begin{enumerate}
\item [\rm(1)]If $T$ is $C_{u}$-symmetric then it has a symbol of the
form
\[
\left[\begin{array}{cc}
a & b\\
c & a-(\overline{u}b+uc)/\Delta
\end{array}\right]
\]
for some $a,b,c\in L^{2}$.
\item [\rm(2)]If $T$ is $C_{u}$-skew-symmetric then it has a symbol of
the form
\[
\left[\begin{array}{cc}
-\Delta f & uf\\
\overline{u}f & \Delta f
\end{array}\right]
\]
for some $f\in L^{2}$.
\end{enumerate}
\end{prop}
\begin{proof}
By Theorem \ref{teorema principala} and Proposition \ref{prop:limit for uniqueness}, $T$ has a symbol of the form
\[
G=\left[\begin{array}{cc}
\alpha & \beta\\
\gamma & \delta
\end{array}\right]
\]
with $\alpha,\beta,\gamma\in L^{2}$ and $\delta\in L^{\infty}.$
If $T$ is $C_{u}$ symmetric, the function $F$ such that $M_{F}=(1/2)(M_{G}+CM_{G}^{*}C)$
is again a symbol for $T$ and it has the form specified in (1) by
Proposition of \ref{prop:MF symmetric}. If $T$ is $C_{u}$-skew-symmetric,
we use instead the operator $M_{H}=(1/2)(M_{G}-CM_{G}^{*}C)$. The
proof of Proposition \ref{prop:MF symmetric} shows that $H$ has
the form specified in (2).\end{proof}

\end{document}